\documentclass[11pt]{article}
\usepackage{mathrsfs}

\usepackage{amsthm}
\usepackage{amsmath}
\usepackage{graphicx}
\usepackage{epstopdf}
\usepackage{titling}

\newtheorem{theorem}{Theorem}[section]
\newtheorem{proposition}[theorem]{Proposition}
\newtheorem{lemma}[theorem]{Lemma}

\newtheorem{corollary}[theorem]{Corollary}

\begin{document}

\title{Large feedback arc sets, high minimum degree subgraphs, and long cycles in
Eulerian digraphs}

\author{Hao Huang\thanks{Department of Mathematics, UCLA, Los
Angeles, CA 90095. Email: {\tt huanghao@math.ucla.edu}. Research
supported by a UC Dissertation Year Fellowship.} \and Jie
Ma\thanks{Department of Mathematics, UCLA, Los Angeles, CA 90095.
Email: {\tt jiema@math.ucla.edu}.} \and Asaf Shapira \thanks{School of Mathematics, Tel-Aviv University, Tel-Aviv, Israel 69978,
and Schools of Mathematics and Computer Science, Georgia Institute of Technology, Atlanta, GA 30332. Email: {\tt asafico@math.gatech.edu}.
Supported in part by NSF Grant DMS-0901355 and ISF Grant 224/11.} \and Benny
Sudakov\thanks{Department of Mathematics, UCLA, Los Angeles, CA
90095. Email: {\tt bsudakov@math.ucla.edu}. Research supported in
part by NSF grant DMS-1101185, NSF CAREER award DMS-0812005, and by
a USA-Israeli BSF grant.} \and Raphael Yuster
\thanks{Department of Mathematics, University of Haifa, Haifa
31905, Israel. Email: {\tt raphy@math.haifa.ac.il}.}}

\date{}

\maketitle

\setcounter{page}{1}

\vspace{-2em}
\begin{abstract}

A {\em minimum feedback arc set} of a directed graph $G$ is a
smallest set of arcs whose removal makes $G$ acyclic. Its
cardinality is denoted by $\beta(G)$. We show that an Eulerian
digraph with $n$ vertices and $m$ arcs has $\beta(G) \ge
m^2/2n^2+m/2n$, and this bound is optimal for infinitely many $m,
n$. Using this result we prove that an Eulerian digraph contains a
cycle of length at most $6n^2/m$, and has an Eulerian subgraph with
minimum degree at least $m^2/24n^3$. Both estimates are tight
up to a constant factor. Finally, motivated by a conjecture of Bollob\'as
and Scott, we also show how to find long cycles in Eulerian digraphs.


\end{abstract}

\section{Introduction}

One of the central themes in graph theory is to study the extremal
graphs which satisfy certain properties. Extremity can be
taken with respect to different parameters as order, size, or girth.
There are many classical results in this area. For example, any
undirected graph $G$ with $n$ vertices and $m$ edges has a subgraph
with minimum degree at least $m/n$, and thus $G$ also
contains a cycle of length at least $m/n+1$. It is natural
to ask whether such results can be extended to digraphs. However, it
turns out that these statements are often trivially false even for very dense
general digraphs. For instance, a transitive tournament does not
contain any cycle, and its subgraphs always have zero minimum
in-degree and out-degree. Therefore in order to obtain meaningful
results as in the undirected case, it is necessary to restrict to a
smaller family of digraphs. A natural candidate one may consider is
the family of {\em Eulerian digraphs}, in which the in-degree equals the out-degree
at each vertex. In this paper, we investigate several
natural parameters of Eulerian digraphs, and study the connections between them.
In particular, the parameters we consider are minimum feedback arc set, shortest cycle , longest cycle,
and largest minimum degree subgraph. Throughout this paper, we
always assume the Eulerian digraph is simple, i.e. it has no
multiple arcs or loops, but arcs in different directions like
$(u,v)$ and $(v,u)$ are allowed. For other standard graph-theoretic
terminology involved, the reader is referred to
\cite{bollobas-modern}.

A {\em feedback arc set} of a digraph is a set of arcs whose removal makes the digraph acyclic.
Given a digraph $G$, denote by $\beta(G)$ the
minimum size of a feedback arc set. Computing $\beta(G)$ and finding a corresponding minimum feedback arc set
is a fundamental problem in combinatorial
optimization. It has applications in many other fields such as testing of electronic
circuits and efficient deadlock resolution (see, e.g., \cite{LS, shaw}).
However, computing $\beta(G)$ turns out to be difficult, and it
is NP-hard even for tournaments \cite{alon-np,charbit}.
One basic question in this area is to bound $\beta(G)$ as a function of other parameters of $G$,
and there are several papers (see, e.g., \cite{CSS,FKS,sullivan-phd}) studying upper bounds for $\beta(G)$ of this form.
However, much less is known for the lower bound of $\beta(G)$, perhaps because a general digraph could be very dense
and still have a small minimum feedback arc set. For example, a transitive tournament has $\beta(G)=0$.
Nevertheless, it is easy to see that any Eulerian digraph $G$
with $n$ vertices and $m$ arcs has $\beta(G) \geq m/n$, since the
arcs can be decomposed into a disjoint union of cycles, each of
length at most $n$, and any feedback arc set contains
at least one arc from each cycle. In this paper we actually prove
the following much stronger lower bound for $\beta(G)$.

\begin{theorem}
\label{t:main-fas}
Every Eulerian digraph $G$ with $n$ vertices and $m$ arcs has $\beta(G) \ge m^2/2n^2+ m/2n$.
\end{theorem}

Moreover, Theorem \ref{t:main-fas} is tight for an infinite family
of Eulerian digraphs, as can be seen from the following proposition.
\begin{proposition}
\label{p:fas-construct}
For every pair of integers $m$ and $n$ such that $m$ is divisible by $n$, there exists an Eulerian digraph $G$ with $n$
vertices and $m$ arcs, and with $\beta(G) = m^2/2n^2+m/2n$.
\end{proposition}

The study of the existence of cycles plays a very important role in
graph theory, and there are numerous results for undirected graphs
in the classical literature. However, there are significantly fewer
results for digraphs. The main reason for this is probably because
digraphs often behave more similar to hypergraphs, and questions
concerning cycles in digraphs are often much more difficult than the
corresponding questions in graphs. One of the most famous problems
in this area is the celebrated Caccetta-H\"{a}ggkvist conjecture
\cite{c-h}: every directed $n$-vertex digraph with minimum outdegree
at least $r$ contains a cycle with length at most $\lceil n/r
\rceil$, which is not completely solved even when restricted to Eulerian
digraphs (for more discussion, we direct the interested reader to
the surveys \cite{nathanson,sullivan}). In this paper, we study the existence
of short cycles in Eulerian digraphs with a given order and size.
The {\em girth} $g(G)$ of a digraph $G$ is defined as the length of the shortest cycle in $G$.
Combining Theorem \ref{t:main-fas} and a result of Fox, Keevash and
Sudakov \cite{FKS} which connects $\beta(G)$ and $g(G)$ for a general digraph $G$, we are able to obtain the
following corollary.

\begin{corollary}
\label{c:girth} Every Eulerian digraph $G$ with $n$ vertices and $m$
arcs has $g(G) \le 6n^2/m.$
\end{corollary}
We also point out that the upper bound in Corollary \ref{c:girth}
is tight up to a constant, since the construction of Proposition
\ref{p:fas-construct} also provides an example of Eulerian digraphs
with girth at least $n^2/m$.

A repeated application of Corollary \ref{c:girth}
gives an Eulerian subgraph of the original digraph $G$, whose arc set is a disjoint union of $\Omega(m^2/n^2)$ cycles. Using this fact we
can find an Eulerian subgraph of $G$ with large minimum degree.

\begin{theorem}
\label{t:subgraph} Every Eulerian digraph $G$ with $n$ vertices and
$m$ arcs has an Eulerian subgraph with minimum degree at least
$m^2/24n^3$. This bound is tight up to a constant for infinitely
many $m,n$.
\end{theorem}

In 1996, Bollob\'as and Scott (\cite{bollobas-scott}, Conjecture
$6$) asked whether every Eulerian digraph $G$ with nonnegative arc-weighting
$w$ contains a cycle of weight at least $cw(G)/(n-1)$, where $w(G)$
is the total weight and $c$ is some absolute constant. For the
unweighted case, i.e. $w = 1$, this conjecture becomes: ``Is it true
that every Eulerian digraph with $n$ vertices and $m$ arcs contains
a cycle of length at least $cm/n$?'' Even this special case is still
wide open after 15 years. An obvious consequence of Theorem
\ref{t:subgraph} is that every Eulerian digraph contains a cycle of
length at least $1+m^2/24n^3$. When the digraph is dense, i.e.
$m=cn^2$, our theorem provides a cycle of length linear in $n$,
which partially verifies the Bollob\'as-Scott conjecture in this
range. However observe that when $m$ is small, in particular when
$m=o(n^{3/2})$, Theorem \ref{t:subgraph} becomes meaningless.
Nevertheless, we can always find a long cycle of length at least
$\lfloor \sqrt{m/n} \rfloor + 1$, as shown by the following proposition\footnote{This proposition was also obtained
independently by Jacques Verstraete.}.

\begin{proposition}
\label{p:circumference} Every Eulerian digraph $G$ with $n$ vertices
and $m$ arcs has a cycle of length at least $1+ \lfloor \sqrt{m/n} \rfloor $.
Together with Theorem \ref{t:subgraph}, this implies that $G$ has
a cycle of length at least $1+\max\{m^2/24n^3,\lfloor \sqrt{m/n} \rfloor\}$.
\end{proposition}

The rest of this paper is organized as follows. In Section 2, we
obtain our bounds for feedback arc sets by proving Theorem
\ref{t:main-fas} and Proposition \ref{p:fas-construct}. Section 3
contains the proof of our results for the existences of short
cycles, long cycles, and subgraph with large minimum degree. The
final section contains some concluding remarks and open problems.

\section{Feedback arc sets}

This section contains the proofs of Theorem \ref{t:main-fas} and
Proposition \ref{p:fas-construct}. Consider some linear order  of
the vertex set of an Eulerian digraph $G=(V,E)$ with $n$ vertices and $m$ arcs. Let $v_i$ is the $i$'th
vertex in this order. We say that $v_i$ is {\em before} $v_j$ if $i
< j$. An arc $(v_i,v_j)$ is a {\em forward} arc if $i < j$, and is a
{\em backward} arc if $i
> j$. Observe that any cycle contains at least one backward arc.
Hence, the set of backward arcs forms a feedback arc set. We prove
Theorem \ref{t:main-fas} by showing that any linear order of $V$ has
at least as many backward arcs as the amount stated in the theorem.
We first require the following simple lemma. Here a {\em cut} is
defined as a partition of the vertices of a digraph into two
disjoint subsets.
\begin{lemma}
\label{l:cut} In any cut $(A,V \setminus A)$ of an Eulerian digraph,
the number of arcs from $A$ to $V \setminus A$ equals the number of
arcs from $V \setminus A$ to $A$.
\end{lemma}
\begin{proof}
The sum of the out-degrees of the vertices of $A$ equals the sum of the in-degrees of the vertices of $A$. Each arc with both endpoints in $A$
contributes one unit to each of these sums. Hence, the number of arcs with only one endpoint in $A$ splits equally between arcs that go from $A$
to $V \setminus A$ and arcs that go from $V \setminus A$ to $A$.
\end{proof}

\noindent {\bf Proof of Theorem \ref{t:main-fas}.}~
Fix an Eulerian digraph $G$ with $|V|=n$ and $|E|=m$. We claim that it suffices to
only consider Eulerian digraphs which are $2$-cycle-free, i.e.
between any pair of vertices $\{i, j\}$, there do not exist arcs in
two different directions. Suppose there are $k$ different $2$-cycles
in $G$. By removing all of them, we delete exactly $2k$ arcs. Note
that the resulting $2$-cycle-free digraph $G'$ is still Eulerian and
contains $m-2k$ arcs. Therefore if Theorem \ref{t:main-fas} is true
for all $2$-cycle-free Eulerian digraphs, then
$$\beta(G')\ \ge \frac{(m-2k)^2}{2n^2} + \frac{m-2k}{2n}\;.$$
Obviously in any linear order of $V(G)$, exactly half of the $2k$
arcs deleted must be backward arcs. Therefore,
\begin{align*}
\beta(G) &\ge \beta(G') + k \ge \frac{(m-2k)^2}{2n^2} + \frac{m-2k}{2n} + k = \left(\frac{m^2}{2n^2}+\frac{m}{2n}\right) - \frac{2k(m-k)}{n^2} + k - \frac{k}{n}\\
&\ge \left(\frac{m^2}{2n^2}+\frac{m}{2n}\right) - \frac{2k \binom{n}{2}}{n^2} + k - \frac{k}{n} =\frac{m^2}{2n^2}+\frac{m}{2n}\;.
\end{align*}
The last inequality follows from the fact that $m-k \leq \binom{n}{2}$,
since $m-k$ counts the number of pairs of vertices with an arc between them.

From now on, we always assume that $G$ is a $2$-cycle-free Eulerian
digraph. In order to prove a lower bound on $\beta(G)$, we fix a
linear order $v_1,\ldots,v_n$. It will be important for the analysis
to consider the {\em length} of an arc $(v_i,v_j)$ which is $|i-j|$.
Observe that the length of any arc is an integer in
$\{1,\ldots,n-1\}$. Moreover, we call an arc {\em short} if its length is
at most $n/2$. Otherwise, it is {\em long}.

Partition the arc set $E$ into two parts, $S$ and $L$, where $S$ contains the
short arcs and $L$ contains the long arcs. For a vertex $v_i$, let
$s_i$ denote the number of short arcs connecting $v_i$ with some
$v_j$ where $j > i$. It is important to note that at this point we
claim nothing regarding the directions of these arcs. Since $G$ is
$2$-cycle-free, $s_i \leq n-i$. As each short arc $(v_i,v_j)$
contributes exactly one to either $s_i$ or $s_j$, we have that:
\begin{equation*}
\sum_{i=1}^n s_i = |S|\;.
\end{equation*}

We now estimate the sum of the lengths of the short arcs. Consider
some vertex $v_i$. Since $G$ is $2$-cycle-free, the $s_i$ short arcs
connecting $v_i$ to vertices appearing after $v_i$ must have
distinct lengths. Hence, the sum of their lengths is at least
$1+2+\cdots+s_i = {{s_i+1} \choose 2}$. Thus, denoting by $w(S)$ the
sum of the lengths of the short arcs, we have that
\begin{equation}\label{weight_short}
w(S) \ge \sum_{i=1}^n {{s_i+1} \choose 2}\;.
\end{equation}

Next we calculate the sum of the lengths of the long arcs, that is
denoted by $w(L)$. There is at most one long arc of length $n-1$.
There are at most two arcs of length $n-2$, and, more generally,
there are at most $n-i$ arcs of length $i$. Thus, if we denote by
$t_i$ the number of long arcs of length $i$ for $i\ge \lfloor
n/2\rfloor+1$ and set $t_i=0$ for $i\le \lfloor n/2\rfloor$, we have
that $t_i \leq n-i$, and
\begin{equation}\label{weight_long}
w(L) = \sum_{i=1}^n i \cdot t_i\;.
\end{equation}
Obviously,
$$\sum_{i=1}^n t_i + \sum_{i=1}^n s_i =|L| + |S| = m\;.$$
Let $A_i=\{v_1,\ldots,v_i\}$ and consider the cuts $C_i=(A_i,V
\setminus A_i)$ for $i=1,\ldots,n$. Let $c_i$ denote the number of
arcs crossing $C_i$ (and notice that $c_n=0$). Since an arc of
length $x$ crosses precisely $x$ of these cuts, we have that
\begin{equation}\label{e:sumci}
\sum_{i=1}^{n} c_i = w(S) + w(L)\;.
\end{equation}
Consider a pair of cuts $C_i,C_{i+\lfloor n/2 \rfloor}$ for $i=1,\ldots,\lfloor n/2 \rfloor$.
If an arc crosses both $C_i$ and $C_{i+\lfloor n/2 \rfloor}$ then its length is at least $\lfloor n/2 \rfloor + 1$.
Hence, a short arc cannot cross both of these cuts. Let $y_i$ denote the number of long arcs that cross
both of these cuts. By Lemma \ref{l:cut}, $c_i/2$ backward arcs cross $C_i$ and $c_{i+\lfloor n/2 \rfloor}/2$
backward arcs cross $C_{i+\lfloor n/2 \rfloor}$, and we have counted at most $y_i$ such arcs twice.
It follows that the number of backward arcs is at least
$$
\frac{1}{2}(c_i + c_{i+\lfloor n/2 \rfloor})-y_i\;.
$$
Averaging over all $\lfloor n/2 \rfloor$ such pairs of cuts, it follows that the number of backward arcs is at least
\begin{equation}
\label{e:avg}
\frac{1}{\lfloor n/2 \rfloor} \sum_{i=1}^{\lfloor n/2 \rfloor} \left(\frac{1}{2}(c_i + c_{i+\lfloor n/2 \rfloor})-y_i\right)\;.
\end{equation}
As each long arc of length $j$ crosses precisely $j-\lfloor n/2 \rfloor$ pairs of cuts $C_i$ and $C_{i+\lfloor n/2 \rfloor}$,
we have that $\sum_{i=1}^{\lfloor n/2 \rfloor} y_i = \sum_{j \geq \lfloor n/2 \rfloor} t_j (j-\lfloor n/2 \rfloor) = w(L) - |L| \cdot \lfloor n/2 \rfloor$.  This, together with
(\ref{e:sumci}) and (\ref{e:avg}) gives that
\begin{eqnarray} \label{e:generalbound}
\beta(G) & \ge & \frac{1}{\lfloor n/2 \rfloor} \left (\frac{1}{2}(w(S)+w(L))-(w(L)- |L| \cdot \lfloor n/2 \rfloor) \right) \nonumber\\
& \ge & \frac{w(S)-w(L)}{2\lfloor n/2 \rfloor} + |L|\;.
\end{eqnarray}
Note that when $n=2k$ is even, the above inequality becomes
\begin{equation*}
\beta(G) \ge \frac{w(S)-w(L)}{n} + |L|\;.
\end{equation*}
Next we show that when $n=2k+1$ is odd, the same inequality still holds. To see this, first assume that
$w(S) \geq w(L)$. Then applying inequality \eqref{e:generalbound}, we have that for $n=2k+1$,
\begin{equation*}
\beta(G) \ge \frac{w(S)-w(L)}{2k} + |L| \geq \frac{w(S)-w(L)}{n} + L\;.
\end{equation*}
Next suppose that $w(S)< w(L)$. Instead of considering the cuts $C_i$ and $C_{i+k}$, we look at
the pair $C_i$ and $C_{i+k+1}$ for $i=1, \cdots, k$. Moreover,
denote by $z_i$ the number of long arcs that cross both of these
cuts. By a similar argument as before, the number of backward arcs is
at least $\frac{1}{2}(c_i+c_{i+k+1})-z_i$ for $1 \le i \le k$, and
$c_{i}/2$ for $i=k+1$. This provides $k+1$ lower bounds for $\beta(G)$, and we
will average over all of them.
Since each long arc of length $j$ crosses
precisely $j-(k+1)$ pairs of cuts $C_i$ and $C_{i+k+1}$, we again have that
$\sum_{i=1}^{k} z_i = \sum_{j \geq k+1} t_j (j-(k+1)) =  w(L) - (k+1)|L|$,
and we have that
\begin{eqnarray*}
\beta(G) & \ge & \frac{1}{k+1} \left(\sum_{i=1}^k \left( \frac{1}{2}(c_i + c_{i+k+1})-z_i\right )+\frac{c_{k+1}}{2} \right) \nonumber \\
& \ge & \frac{1}{k+1} \left (\frac{1}{2}(w(S)+w(L))-(w(L)- (k+1)|L|) \right) \nonumber\\
& \ge & \frac{w(S)-w(L)}{2k+2} + |L| \ge \frac{w(S)-w(L)}{n} + |L|\;,
\end{eqnarray*}
where we use the fact that $w(L)>w(S)$.

Using our lower bound estimate \eqref{weight_short} for $w(S)$ and the expression \eqref{weight_long} for $w(L)$, we obtain that
\begin{eqnarray}\label{e:fes-gamma}
\beta(G)  &\ge&  \frac{w(S)-w(L)}{n} + |L| \nonumber\\
& \ge & \frac{1}{n} \left (\sum_{i=1}^n \binom{s_i+1}{2} - \sum_{i=1}^n i \cdot t_i \right) + \sum_{i=1}^n t_i \\
& = & \frac{1}{n} \left (\sum_{i=1}^n \binom{s_i+1}{2} + (n-i)t_i \right) \nonumber\;.
\end{eqnarray}
Define
$$F(s_1, \cdots, s_n; t_1, \cdots, t_n) := \sum_{i=1}^n \binom{s_i+1}{2} + (n-i)t_i\;.$$
In order to find a lower bound of $\beta(G)$, we need to solve the following integer optimization problem.
\begin{align*}
F(m,n) &:= \min F(s_1, \cdots, s_n; t_1, \cdots, t_n)\\
\text{subject to } &s_i \leq n-i, \quad t_i \leq n-i, \quad \sum_{i=1}^n s_i + \sum_{i=1}^n t_i = m\;.
\end{align*}
The following Lemma \ref{optimization} provides a precise solution to this optimization problem, which
gives that $F(m,n) = tm - (t^2-t)n/2$, where $t=\lceil m/n \rceil$. Hence if we assume
that $m = tn-k$ with $0 \leq k \leq n-1$, then
\begin{align*}
\beta(G) &\geq \frac{1}{n}F(m,n) = \frac{tm}{n}-\frac{t^2-t}{2}=\frac{t(tn-k)}{n}-\frac{t^2-t}{2}\\
&=\frac{t^2+t}{2}- \dfrac{tk}{n} \geq \frac{t^2+t}{2}- \dfrac{tk}{n}+ \left (\frac{k^2}{2n^2}-\frac{k}{2n} \right)\\
&=\frac{(tn-k)^2}{2n^2}+\frac{tn-k}{2n} = \frac{m^2}{2n^2}+\frac{m}{2n}\;.
\end{align*}
The last inequality is because $0 \leq k \leq n-1$, so $0 \le k/n<1$ and $k^2/2n^2 \leq k/2n$.
Note that equality is possible only when $m$ is a multiple of $n$.\qed


\begin{lemma}\label{optimization}
$F(m,n) = tm-(t^2-t)n/2$, where $t=\lceil m/n \rceil$.
\end{lemma}
\begin{proof}
The proof of this lemma consists of several claims.
We assume that $s_i+t_i=a_i$, then $0 \leq a_i \leq 2(n-i)$ and $s_i \leq n-i$, so
$$\binom{s_i+1}{2}+(n-i)t_i = \frac{1}{2}s_i^2-(n-i-1/2)s_i+(n-i)a_i\;.$$
Since $s_i$ is an integer, this function of $s_i$ is minimized when $s_i=n-i$ if $a_i \geq n-i$, and when $s_i=a_i$ if $a_i<n-i$.
Therefore, subject to $\sum_i a_i = m$ and $a_i \leq 2(n-i)$, we want to minimize
\begin{align}\label{e:F}
F &= \sum_{a_i<n-i} \binom{a_i+1}{2} + \sum_{a_i \geq n-i} \left(\binom{n-i+1}{2} +(n-i)(a_i-(n-i))\right)\nonumber \\
&=\sum_{a_i<n-i} \binom{a_i+1}{2} + \sum_{a_i \geq n-i} \left((n-i)a_i - \binom{n-i}{2}\right)\;.
\end{align}
For convenience, define $A= \{i: a_i<n-i\}$, and $B=\{i: a_i \geq
n-i\}$.
\medskip

\noindent {\bf Claim 1.} For any $i \in A$, if we increase $a_i$ by $1$ then $F$ increases by $a_i+1$,
and if we decrease $a_i$ by $1$ then $F$ decreases by $a_i$.
For any $j \in B$, if we increase (decrease) $a_j$ by $1$ then $F$ increases (decreases) by $n-j$.
\begin{proof}
Note that when $a_i=n-i$ or $n-i-1$,
$\binom{a_i+1}{2}=(n-i)a_i-\binom{n-i}{2}$, therefore
if we increase $a_i$ by $1$ for any
$i \in A$, the contribution of $a_i$ to $F$ always increases by
$\binom{a_i+2}{2}-\binom{a_i+1}{2} = a_i+1$.
When we decrease $a_i$ by $1$, $F$ decreases by $\binom{a_i+1}{2}-\binom{a_i}{2} =a_i$.
It is also easy to see that for any $j \in B$, if we increase or decrease
$a_j$ by $1$,  the contribution of $a_j$ to $F$
always increases or decreases by $n-j$.
\end{proof}

Next we show that for any extremal configuration $(a_1, \cdots, a_n)$ which minimizes $F$,
any integer of $A$ is smaller than any integer of $B$.\vspace{4pt}

\noindent {\bf Claim 2.} $F$ is minimized when $A=\{1, \cdots, l-1\}$ and $B=\{l,\cdots, n\}$ for some integer $l$.
\begin{proof}
We prove by contradiction. Suppose this statement is false, then $F$
is minimized by some $\{a_i\}_{i=1}^n$ such that there exists $i<j$,
$i \in B$ and $j \in A$. Now we decrease $a_i$ by $1$ and
increase $a_j$ by $1$, which can be done since $a_j<2(n-j)$. Then by Claim $1$, $F$ decreases by
$(n-i)-(a_j+1) \geq n-(j-1) - (a_j+1) = (n-j) - a_j > 0$ since $j \in A$,
which contradicts the minimality of $F$.
\end{proof}

Since $\sum_{i=1}^n a_i = m$, which is fixed. The next claim shows that in order to minimize $F$, we need to take
the variables whose index is in $B$ to be as large as possible, with at most one exception.\vspace{4pt}

\noindent {\bf Claim 3.} $F$ is minimized when $A=\{1, \cdots, l-1\}$, and $B=\{l,\cdots, n\}$ for some integer $l$. Moreover, $a_i = 2(n-i)$ for all $i \geq l+1$.
\begin{proof}
First note that for $i \in B$, its contribution to $F$ is $(n-i)a_i-\binom{n-i}{2}$. The second term is fixed, and $a_i$ has coefficient
$n-i$ which decreases in $i$. Therefore when $F$ is minimized, if
$i$ is the largest index in $B$ such that $a_i < 2(n-i)$, then all $j<i$
in $B$ must satisfy $a_j = n-j$; otherwise
we might decrease $a_j$ and increase $a_i$ to make $F$ smaller.
Therefore, if $i>l$, we have $a_{i-1}=n-i+1$. Note that if we increase
$a_i$ by $1$ and decrease $a_{i-1}$ by $1$, by Claim $1$ the target function $F$ decreases by
$a_{i-1}-(n-i)=1$. Therefore the only
possibility is that $i=l$, which proves Claim $3$.
\end{proof}

\noindent {\bf Claim 4.} There is an extremal configuration for which $a_i=n-l$ or $a_i=n-l+1$ for $i \leq l-1$,
$a_l$ is between $n-l$ and $2(n-l)$, and $a_i=2(n-i)$ for $i \geq l+1$.
\begin{proof}
From Claim $3$, we know that in an extremal configuration,
$a_i<n-i$ for $1 \leq i \leq l-1$, $n-l \leq a_l \leq 2(n-l)$, and
$a_i = 2(n-i)$ for $i \geq l+1$. Among all extremal
configurations, we take one with the largest $l$, and for all such configurations, we take one for which $a_l$ is the smallest.
For such a configuration, if we increase $a_j$ by $1$ for some $j\in A$ and
decrease $a_l$ by $1$, then by Claim $1$, $F$ increases by $(a_j+1)-(n-l)$, which must
be nonnegative. Suppose $a_j+1=n-l$. If $j$ is changed to be in $B$,
it contradicts Claim $3$ no matter whether $l$ remains in $B$ or is
changed to be in $A$; if $j$ remains in $A$, it contradicts the
maximality of $l$ if $l$ is changed to be in $A$ or contradicts the
minimality of $a_l$ if $l$ remains in $B$. Therefore $a_j \geq n-l$
for every $1 \le j \le l-1$. We next consider two cases: either $a_l$
is equal to $2(n-l)$, or strictly less than $2(n-l)$.

\medskip

\underline{Case 1.} $a_l = 2(n-l)$. From the discussions above, we already know that $a_j \geq n-l$ for every $1 \le j \le l-1$. In particular $a_{l-1} = n-l$ since it is strictly less than $n-(l-1)$.
If for some $j \leq l-1$, $a_j \geq n-l+2$, then we can decrease $a_j$ by $1$ and increase $a_{l-1}$ by $1$ since $a_j$ is strictly greater than $0$ and $a_{l-1}$ is
strictly less than $2(n-l+1)$.
By Claim $1$, $F$ decreases by $a_j-(n-l+1) \geq 1$, which contradicts the minimality of $F$.
Hence we have that $n-l \le a_j \leq n-l+1$ for every $j \leq l-1$.

\medskip

\underline{Case 2.} $a_l < 2(n-l)$. If we decrease $a_j$ by $1$ and increase $a_l$ by $1$, $F$ decreases by $a_j-(n-l)$ by Claim $1$, therefore $a_j \leq n-l$ by the minimality of $F$, hence $a_j=n-l$ for all $1 \leq j \leq l-1$.

\medskip

In both cases, the extremal configuration consists of $n-l$ or $n-l+1$ for the first $l-1$ variables, $a_l$ is between $n-l$ and $2(n-l)$, and $a_i=2(n-i)$ for $i \geq l+1$.
\end{proof}
By Claim $4$, we can bound the number of arcs $m$ from both sides,
\begin{align*}
m & = \sum_{i=1}^{l-1} a_i + \sum_{i=l}^n a_i \geq (l-1)(n-l) + (n-l) + \sum_{i=l+1}^n 2(n-i)=(n-l)(n-1)\;.
\end{align*}
\vspace{-0.5cm}
\begin{align*}
m & = \sum_{i=1}^{l-1} a_i + \sum_{i=l}^n a_i < (l-1)(n-l+1) + \sum_{i=l}^n 2(n-i) = (n-l+1)(n-1)\;.
\end{align*}
Solving these two inequalities, we get
$$ n-\dfrac{m}{n-1} \leq l < n+1-\dfrac{m}{n-1}\;.$$

Let $m=tn-k$, where $t=\lceil m/n \rceil$ and $0\le k \le n-1$.
It is not difficult to check that if $t\ge k$, $l=n-t$ and if $t<k$, $l=n-t+1$.

Now let $x$ be the number of variables $a_1,...,a_{l-1}$ which are equal to $n-l+1$.
Since $a_i = 2(n-i)$ for $i \ge l+1$, we have that
\begin{equation}\label{distribution}
x+a_l=m-(l-1)(n-l)-\sum_{i\ge l+1} a_i= m-(n-2)(n-l)\;.
\end{equation}

When  $t\ge k$, then $l=n-t$ and
$$x+a_l=m-(n-2)t=2t-k < 2t = 2(n-l),$$
hence $a_l<2(n-l)$. By the analysis of the second case in Claim $4$,
$a_j = n-l = t$ for all $j\le l-1$, therefore $x=0$ and $a_l=2t-k$.
Since $l=n-t$, then using the summation formula $\sum_{k=1}^n k^2= k(k+1)(2k+1)/6$, we have from (\ref{e:F}) that (with details of the calculation omitted)
$$F={t+1\choose 2} (n-t-1)+ t(2t-k)-{t\choose 2} +\sum_{i\ge
l+1} \left( 2(n-i)^2-{n-i\choose 2} \right)= tm-(t^2-t)n/2\;.$$

Now we assume $t<k$, then $l=n-t+1$. Then using \eqref{distribution} again,
$$x+a_l=m-(n-2)(t-1)=n-k+2(t-1)>2(t-1)=2(n-l)\;.$$
The only possibility without contradicting the second case in Claim $4$ is that $a_l=2(n-l)$ and
$x=n-k$. Thus, there are $n-k$ of $a_1,...,a_{l-1}$ which are equal to $n-l+1=t$ and the
rest $k-t$ are equal to $t-1$. Again by (\ref{e:F}), $$F={t+1\choose 2}(n-k)+{t\choose 2}(k-t)+\sum_{i\ge
l} \left( 2(n-i)^2-{n-i\choose 2} \right)= tm-(t^2-t)n/2\;.$$  As we have covered both cases, we have completed the proof of Lemma \ref{optimization}.
\end{proof}
\medskip

\noindent {\bf Proof of Proposition \ref{p:fas-construct}.}~~Now we
construct an infinite family of Eulerian digraphs which achieve the bound in Theorem \ref{t:main-fas}.
For any positive integers
$n,m$ such that $t:=m/n$ is an integer, we define the Cayley digraph $G(n,m)$
to have vertex set $\{1,2,...,n\}$ and arc set $\{(i,i+j): 1\le i\le
n, 1\le j\le t\}$, where all additions are modulo $n$. From the definition, it is easy to verify that
$G(n,m)$ is an Eulerian digraph. Consider an order of the vertex set
such that vertex $i$ is the $i$'th vertex in this order, we observe
that for $n-t+1 \le i\le n$, vertex $i$ has backward arcs
$(i,j)$, where $1+n-i\le j\le t$ and there is no backward arc from vertex $i$ for $i \leq n-t$.
Therefore, $$\beta(G(n,m))\le
\sum_{i=n-t+1}^n t-(n-i) =\sum_{j=1}^t j =
{t+1\choose 2}= \frac{m^2}{2n^2}+ \frac{m}{2n}\;.$$
\qed

\section{Short cycles, long cycles, and Eulerian subgraphs with high minimum degree}

In this section, we prove the existence of short cycles, long cycles,
and subgraphs with large minimum degree in Eulerian digraphs.
An important component in our proofs is the following result by Fox, Keevash and Sudakov \cite{FKS} on general digraphs.
We point out that the original Theorem 1.2 in \cite{FKS} was proved with a constant $25$, which can be improved to $18$ using the exact same proof if we further assume
$r \geq 11$.

\begin{theorem} \label{FKS}
If a digraph $G$ with $n$ vertices and $m$ arcs has $\beta(G) > 18n^2/r^2$, with $r \geq 11$,
then $G$ contains a cycle of length at most $r$, i.e. $g(G) \leq r$.
\end{theorem}

Applying this theorem and Theorem \ref{t:main-fas}, we can now prove Corollary \ref{c:girth}, which
says that every Eulerian digraph $G$ with $n$ vertices and $m$ arcs contains a cycle
of length at most $6n^2/m$.

\medskip

\noindent {\bf Proof of Corollary \ref{c:girth}.}~
Given an Eulerian digraph $G$ with $n$ vertices and $m$ arcs, if $G$ contains a $2$-cycle, then $g(G) \leq 2 \leq 6n^2/m$. So
we may assume that $G$ is $2$-cycle-free and thus $m \leq \binom{n}{2}$.
By Theorem \ref{t:main-fas},
$$\beta(G) \geq  \frac{m^2}{2n^2}+\frac{m}{2n} > \frac{m^2}{2n^2} = \frac{18n^2}{(6n^2/m)^2}\;.$$

Since $r=6n^2/m>6n^2/\binom{n}{2}>11$, we can use Theorem \ref{FKS} to conclude that
$$g(G) \leq r = \frac{6n^2}{m}\;.$$
To see that this bound is tight up to a constant factor, we consider the
construction of the Cayley digraphs in Proposition
\ref{p:fas-construct}. It is not hard to see that if $k=m/n$, the shortest directed cycle in $G(n,m)$ has length
at least $\lceil n/k \rceil \geq n^2/m$. \qed

\medskip

Next we show that every Eulerian digraph with $n$ vertices and $m$ arcs has
an Eulerian subgraph with minimum degree $\Omega(m^2/n^3)$.
\medskip

\noindent {\bf Proof of Theorem \ref{t:subgraph}.}~ We start with an Eulerian digraph $G$ with $n$ vertices and $m$ arcs.
Note that Corollary \ref{c:girth} implies
that every Eulerian digraph with $n$ vertices and at least $m/2$ arcs contains a cycle of length
at most $12n^2/m$. In every step, we pick one such cycle and delete all of its arcs from $G$. Obviously the resulting
digraph is still Eulerian, and this process will continue until there are less than $m/2$ arcs left in the digraph. Therefore through this process we obtain a collection $\mathscr{C}$ of
$t$ arc-disjoint cycles $C_1, \cdots, C_t$, where $t \geq (m-m/2)/(12n^2/m) \geq m^2/24n^2$. Denote by $H$ the union of all these cycles, obviously $H$ is
an Eulerian subgraph of $G$.

If $H$ has minimum degree at least $\lceil t/n \rceil \geq m^2/24n^3$, then we are already done. Otherwise, we repeatedly delete from $H$ any vertex $v$ with degree $d(v) \leq \lceil t/n \rceil-1$, together
with all the $d(v)$ cycles in $\mathscr{C}$ passing through $v$. This process stops after
a finite number of steps. In the end we delete at most $n (\lceil t/n \rceil-1) \leq t-1$ cycles in $\mathscr{C}$, so the resulting digraph $H'$ is nonempty. Moreover, every vertex in $H'$
has degree at least $\lceil t/n \rceil \geq m^2/24n^3$. Since $H'$ is the disjoint union of the remaining cycles, it is also an Eulerian subgraph of $G$, and we conclude the proof of Theorem \ref{t:subgraph}. \qed

\medskip

\noindent {\bf Remark. }The proof of Theorem \ref{t:subgraph} also shows that $G$ contains an Eulerian subgraph with minimum degree $\Omega(m^2/n^3)$ and at least $\Omega(m)$ arcs.

\medskip

To see that the bound in Theorem \ref{t:subgraph} is tight up to a constant, for any
integers $s,t>0$, we construct an Eulerian digraph $H:=H(s,t)$ such that
\begin{itemize}

\item $V(H)=(U_1 \cup \cdots \cup U_s )\cup (V_1\cup...\cup V_t)$, $|U_i|=|V_j|=s$ for $1\le i\le s, 1 \le j \le t$,

\item for any $1 \le i \le t-1$ and vertices $u \in V_i$, $v \in V_{i+1}$, the arc $(u,v) \in E(H)$,

\item for any $1 \le i \le s$ and every vertex $u \in U_i$, there is an arc from $u$ to the $i$'th vertex in $V_1$, and another
arc from the $i$'th vertex in $V_t$ to $u$.
\end{itemize}

\begin{figure}[h]
\begin{center}
\includegraphics[scale=0.6]{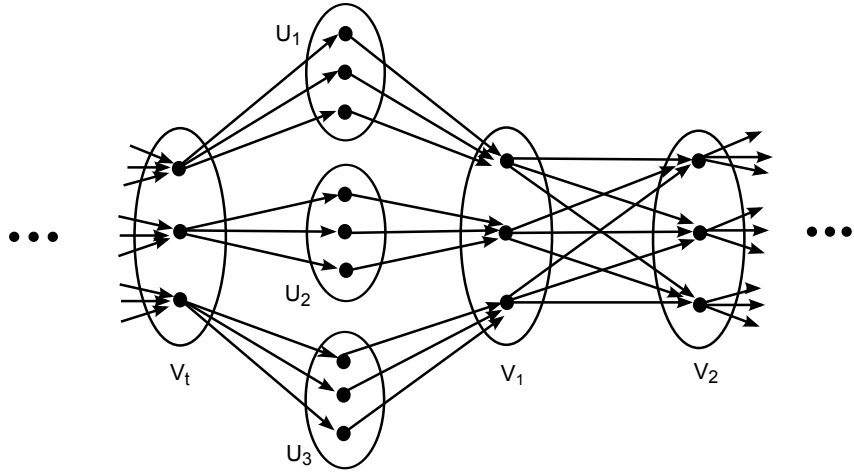}
\end{center}
\caption{The Eulerian digraph $H(s,t)$ with $s=3$}
\label{onlyfigure}
\end{figure}

It can be verified that $H(s,t)$ is an Eulerian digraph with $(s+t)s$ vertices and $s^2(t+1)$ arcs.
Moreover, every cycle in $H(s,t)$ must pass through a vertex in $U_1 \cup \cdots \cup U_s$, whose degree is exactly $1$. Therefore any Eulerian subgraph of $H(s,t)$ has minimum degree at most $1$. Next we define the
$\delta$-blowup $H(s,t,\delta)$: for any integer $\delta>0$, we replace every vertex $i \in V(H(s,t))$ with an independent set $|W_i|=\delta$,
and each arc $(i,j) \in E(H(s,t))$ by a complete bipartite digraph with arcs directed from $W_i$ to $W_j$. The blowup digraph $H(s,t,\delta)$
is still Eulerian, and has $n=s(s+t)\delta$ vertices and $m=s^2(t+1)\delta^2$ arcs.
Taking $t=2s$, we have that for $H(s,2s,\delta)$,
$$\frac{m^2}{n^3}= \frac{(s^2(2s+1)\delta^2)^2}{(s(s+2s)\delta)^3} = \frac{1}{27} \left(2+\frac{1}{s}\right)^2 \delta \geq \frac{4}{27} \delta\;.$$

Note that similarly with the previous discussion on $H(s,t)$, every cycle in the blowup $H(s,2s,\delta)$ contains at least one
vertex with degree $\delta$. Therefore, the minimum degree of any Eulerian subgraph of $H(s,2s,\delta)$ is at most
$\delta \le\frac{27}{4}\frac{m^2}{n^3}$. This implies that the bound in Theorem \ref{t:subgraph} is tight up to a constant factor for infinitely many $m, n$.
\qed

\medskip

Before proving Proposition \ref{p:circumference}, let us recall
the following easy fact.
\begin{proposition}
\label{p:shortcycle} If a digraph $G$ has minimum outdegree $\delta^+(G)$,
then $G$ contains a directed cycle of length at least $\delta^+(G)+1$.
\end{proposition}

\noindent {\bf Proof.} Let $P=v_1 \rightarrow v_2 \rightarrow \cdots \rightarrow v_t$ be the longest directed path
in $G$. Then all the out neighbors of $v_t$ must lie on this path, otherwise $P$ will become longer.
If $i<t$ is minimal with $(v_t, v_i) \in E(G)$, then $v_i \rightarrow \cdots \rightarrow v_t \rightarrow v_i$ gives a cycle of length at least $d^{+}(v_t)+1 \geq \delta^+(G)+1$.\qed

\medskip

This proposition, together with Theorem \ref{t:subgraph}, shows that
an Eulerian digraph $G$ with $n$ vertices and $m$ arcs contains a cycle of length at least
$1+m^2/24n^3$. But as we discussed in the introduction, this bound becomes meaningless when the number of arcs $m$ is small. However, we may use a different approach
to obtain a cycle of length at least $\lfloor \sqrt{m/n} \rfloor+1$.
\medskip

\noindent {\bf Proof of Proposition \ref{p:circumference}.}~ To prove that any Eulerian
digraph $G$ with $n$ vertices and $m$ arcs has a cycle of length at
least $\lfloor \sqrt{m/n} \rfloor +1$, we use induction on the number of vertices $n$.
Note that the base case when $n=2$ is obvious, since the only Eulerian digraph is the $2$-cycle with $\lfloor \sqrt{m/n} \rfloor+1=2$.
Suppose the statement is true for $n-1$. Consider an Eulerian digraph $G$ with $n$ vertices and $m$ arcs.
If its minimum degree $\delta^+(G)\ge \lfloor \sqrt{m/n} \rfloor $, by Proposition \ref{p:shortcycle} $G$ already contains a cycle of length at
least $1+ \lfloor \sqrt{m/n} \rfloor$. Therefore we can assume that there exists a vertex $v$ with $ \lfloor \sqrt{m/n} \rfloor > d^+(v):= t$.
As $G$ is Eulerian, there exist $t$ arc-disjoint cycles $C_1,C_2,...,C_t$ passing through $v$.
If one of these cycles has length at least $\lfloor \sqrt{m/n} \rfloor +1$ then again we are done. Otherwise,
$|C_i|\le \lfloor \sqrt{m/n} \rfloor$ for all $1\le i\le t$. Now we delete from $G$ the vertex $v$ together with the arcs of the cycles $C_1, \cdots, C_t$. The resulting Eulerian digraph
has $n-1$ vertices and $m'$ arcs, where $m'=m-\sum_{i=1}^t |C_i|\ge m-t \lfloor \sqrt{m/n} \rfloor \ge m(1-\frac{1}{n})$. By the inductive hypothesis, the new digraph
(therefore $G$) has a cycle of length at least $1+\sqrt{m'/(n-1)}\ge
1+\sqrt{m(1-\frac{1}{n})/(n-1)} \ge 1+\lfloor \sqrt{m/n} \rfloor$.
\qed

\section{Concluding remarks}

We end with some remarks on the Bollob\'as-Scott conjecture whose unweighted version states that an Eulerian digraph with $n$ vertices and $m$ arcs has
a cycle of length $\Omega(m/n)$. The ``canonical'' proof for showing that an undirected graph with this many vertices and edges has a cycle of length $m/n$ proceeds
by first passing to a subgraph $G'$ with minimum degree at least $m/n$ and then applying Proposition \ref{p:shortcycle} to $G'$. We can then interpret
the second statement of Theorem \ref{t:subgraph} as stating that when applied to Eulerian digraphs, this approach can only produce cycles
of length $O(m^2/n^3)$.

There is, however, another way to show that an undirected graph has a cycle of length $m/n$ using DFS.
Recall that the DFS (Depth First Search) is a graph algorithm that
visits all the vertices of a (directed or undirected) graph $G$ as
follows. It maintains three sets of vertices,  letting $S$ be the
set of vertices which we have completed exploring them, $T$ be the
set of unvisited vertices, and $U = V(G) \setminus (S \cup T)$,
where the vertices of $U$ are kept in a {\em stack} (a {\em last in,
first out} data structure). The DFS starts with $S=U=\emptyset$ and
$T=V(G)$.

While there is a vertex in $V(G) \setminus S$, if $U$ is non-empty,
let $v$ be the last vertex that was added to $U$. If $v$ has a
neighbor $u \in T$, the algorithm inserts $u$ to $U$ and repeats
this step. If $v$ does not have a neighbor in $T$ then $v$ is popped
out from $U$ and is inserted to $S$. If $U$ is empty, the algorithm
chooses an arbitrary vertex from $T$ and pushes it to $U$.
Observe crucially that all the vertices in $U$ form a directed path, and
that there are no edges from $S$ to $T$.

Consider {\em any} DFS tree $T$ rooted at some vertex $v$. Recall that any edge of $G$ is either an edge of $T$ or a backward edge, that is, an edge connecting a vertex $v$ to one of its ancestors in $T$. Hence, if $G$ has no cycle of length at least $t$, then any vertex of $T$ sends at most $t-1$ edges to his ancestors in $T$. This means that $m \leq nt$ or that $t \geq m/n$. Note that this argument shows that {\em any} DFS tree of an undirected graph has depth at least $m/n$. It is thus natural to try and adapt this idea
to the case of Eulerian digraphs. Unfortunately, as the following proposition shows, this approach fails in Eulerian digraphs.

\begin{proposition}\label{DFSdepth} There is an Eulerian digraph $G$ with average degree at least $\sqrt{n}/20$ such that some
DFS tree of $G$ has depth $4$.
\end{proposition}

\begin{proof} We first define a graph $G'$ as follows. Let $t$ be a positive integer
and let $G'$ be a graph consisting of $2t$ vertex sets $V_1,\ldots,V_{2t}$, each of size $t$. We also have a special vertex $r$, so $G'$ has $2t^2+1$ vertices.
We now define the arcs of $G'$ using the following iterative process. We have $t$ iterations,
where in iteration $1 \leq j \leq t$ we add the following arcs; we have $t$ arcs pointing from $r$ to the $t$ vertices of $V_j$, then a matching between the $t$ vertices of $V_j$ to the vertices of $V_{j+1}$, and in general a matching between $V_{k}$ to $V_{k+1}$ for every $j \leq k \leq 2t-j$. We finally have $t$ arcs from $V_{2t-j+1}$ to $r$. We note that we can indeed add a new (disjoint from previous ones) matching between any pair of sets $(V_k,V_{k+1})$ in each of the $t$ iterations by relying on the fact that the edges of the complete bipartite graph $K_{t,t}$ can be split into $t$ perfect matchings.
Observe that in iteration $j$ we add $t(2t-2j+3)$ arcs to $G'$. Hence $G'$ has
$$
\sum^t_{j=1}t(2t-2j+3) \geq t^3
$$
arcs. Moreover it is easy to see from construction that $G'$ is Eulerian.
To get the graph $G$ we modify $G'$ as follows; for every vertex $v \in \bigcup^{2t}_{i=1}V_i$ we add two new vertices $v^{in},v^{out}$ and add a
$4$-cycle $(r,v^{in},v,v^{out},r)$. We get that $G$ has $6t^2+1$ vertices and more than $t^3$ arcs, so setting $n=6t^2+1$ we see that $G$ has average degree at least $\sqrt{n}/20$.

Now consider a DFS tree of $G$ which proceeds as follows; we start at $r$, and then for every $v \in V_{2t}$ go to $v^{in}$ then to $v$ and then to $v^{out}$.
Next, for every $v \in V_{2t-1}$ we go to $v^{in}$ then to $v$ and then to $v^{out}$. We continue this way until we cover all the vertices of $G$. The DFS
tree we thus get has $r$ as its root, and $2t^2$ paths of length $3$ (of type $r,v^{in},v,v^{out}$) attached to it.
\end{proof}

Observe that the above proposition does not rule out the possibility that {\em some} DFS tree has depth $\Omega(m/n)$. We note that proving such a claim
will imply that an Eulerian digraph has a {\em path} of length $\Omega(m/n)$. It appears that even this special case of the Bollob\'as-Scott conjecture
is still open, so it might be interesting to further investigate this problem. In fact, we suspect that if $G$ is a connected Eulerian digraph then for
any vertex $v \in G$ there is a path of length $\Omega(m/n)$ starting at $v$. This statement for undirected graphs follows from the DFS argument at
the beginning of this section.

\bigskip

\noindent \textbf{Acknowledgment:}
We would like to thank Jacques Verstraete for helpful initial discussions.

\end{document}